\documentclass{amsart}
\usepackage[utf8x]{inputenc}
\usepackage{amsmath, amsthm, amssymb}
\usepackage{color}
\usepackage[margin=1.5in]{geometry}
\usepackage{mathtools}
\usepackage{dsfont}
\usepackage[colorlinks=true,linkcolor=blue, citecolor=black]{hyperref}

\makeatletter
\renewcommand*{\eqref}[1]{%
  \hyperref[{#1}]{\textup{\tagform@{\ref*{#1}}}}%
}
\makeatother

\newtheorem{theorem}{Theorem}[section]

\newtheorem{lemma}{Lemma}[section]
\newtheorem*{remark}{Remark}

\newcommand{\R}{\mathbb R}

\newcommand{\1}{\mathds{1}}

\numberwithin{equation}{section}

\title[Eventual regularization for the 3D Muskat problem]{Eventual regularization for the 3D Muskat problem: Lipschitz for finite time implies global existence. }
\author{Stephen Cameron}
\address{Courant Institute, New York University, New York, NY 10012}
\email{spc6@cims.nyu.edu}

\begin{document}

\maketitle

\begin{abstract}
We show that for any fixed Lipschitz constant $L$, there is a time $T^*<\infty$ depending only on $L$ such that if $f:[0,T^*]\times \R^{2}\to [0,1]$ is a classical solution of the stable Muskat problem with $||\nabla_x f||_{L^\infty}\leq L$, then $f$ extends uniquely to a classical solution defined for all time.  Our proof is short, quantitative, and explicit.
\end{abstract}

\section{Introduction}

We study classical solutions $f:[0,T]\times \R^2\to \R$ to the nonlocal parabolic equation 
\begin{equation}\label{e:fequation}
\partial_t f(t,x) = \int\limits_{\R^{2}} \frac{\delta_h f(t,x) - \nabla_x f(t,x)\cdot h}{(\delta_h f(t,x)^2 + |h|^2)^{3/2}} dh,
\end{equation}
where $\delta_h f(t,x) := f(t,x+h)-f(t,x)$ is the partial difference operator.  

The equation arises out of the 3 dimensional incompressible porous media equations
\begin{equation}\label{e:pme}
\left\{\begin{array}{l} \partial_t \varrho + \nabla_X \cdot(u\varrho) = 0, 
\\ u +\nabla_X P + (0,0,\varrho)=0,
\\ \nabla_X \cdot u = 0, 
\\ \varrho(t,X) = \varrho_1 \1_{D(t)}(X) + \varrho_2 \1_{\R^3\setminus D(t)}(X), \end{array}\right.
 \qquad D(t)\subseteq \R^3, \quad (t,X)\in (0,\infty)\times \R^3,
\end{equation}
where $(\varrho, u, P)$ are the fluid density, velocity, and pressure, and we require the density to be piecewise constant.  In the case that the fluid domain $D(t)$ is given as a supergraph 
 $$D(t) = \{X = (x,x_3)\in \R^3| x_3> f(t,x)\},$$ and the heavier fluid is below the lighter fluid (i.e., $\varrho_1 < \varrho_2)$, the system \eqref{e:pme} simplifies to an equation for the interface \eqref{e:fequation} after proper renormalization.  See \cite{Localwellpose} for derivation.

The system \eqref{e:pme} was first considered by Muskat in \cite{Muskat} when he was studying the encroachment of water into tar sand.  While we only study the stable regime in the density jump case, numerous other physically relevant situations can be considered.  The interface can be a more general surface (rather than a graph) \cite{Localwellpose}, the fluids can have different viscosities \cite{ViscosityIllpose, Porous, ViscosityJump}, surface tension effects can be accounted for \cite{Surf1, Surf2, Surf3}, and boundary effects in the finite depth case \cite{FiniteDepth1, FiniteDepth2}.  In the ill posed case where the denser fluid is above the lighter fluid, weak mixing solutions have also been studied \cite{Mixing1, Mixing2}.  

For Lipschitz solutions \eqref{e:fequation} is uniformly parabolic equation of order one.  This ellipticity can also be seen by linearizing around a flat interface, giving the fractional heat equation
\begin{equation}\label{e:linear}
\partial_t f(t,x) = -(-\Delta_x)^{1/2}f(t,x).
\end{equation}

The equation \eqref{e:fequation} was first shown to be locally well-posed in $H^k$ in \cite{Localwellpose} for $k\geq 4$.  Since then, \eqref{e:fequation} has been proven to be locally well-posed in typical subcritical spaces \cite{ParaDiff, ParaLin, WellposeLp}.  It has maximal principles in $L^2$ and $L^\infty$, but unlike in the linear case \eqref{e:linear} these do not imply any gain in regularity.  

One of the most interesting features of \eqref{e:fequation} is that despite being a locally well-posed parabolic equation with maximum principles, it is not globally well-posed.  \cite{Breakdown} proved that wave turning can happen in finite time, causing a failure of the Rayleigh-Taylor condition.  That is, there exists a smooth solution $f$ of \eqref{e:fequation} and time $0<T<\infty$ such that 
$$\displaystyle\lim\limits_{t\to T-} ||\nabla_x f(t,\cdot)||_{L^\infty} = \infty,$$
after which the interface is no longer parametrized as a graph.  Though the evolution of the interface is still well defined as a more general smooth surface for some amount of time after $T$, an example was constructed in \cite{Blowup} with blowup in $C^4$.  Once there is wave turning, the behavior of the interface can become quite complicated and difficult to predict, with solutions having been shown to switch between the stable and unstable regimes multiple times before regularity breaks down \cite{StabilityShift, StabilityShift2}.

As the behavior becomes much more complicated after wave turning, there has been a lot of study of what kinds of assumptions on initial data can rule out this behavior and give us global wellposedness, with most works focusing on the simpler 2D case (one dimensional interface).  Solutions to the Muskat equation \eqref{e:fequation} are preserved under the rescaling $r^{-1}f(r t, r x)$, so these assumptions have typically taken the form of upper bounds on $f_0$ in scaling invariant norms.  

In 2D, \cite{Globalold} shows a global classical solution exists whenever the initial data satisfies $||f_0||_1 = ||\ |\xi| \hat{f}_0(\xi)||_{L_\xi^1(\R)} $ less than some explicit constant, which was improved to $\approx 1/3$ in \cite{3DSlope}.  In this case \cite{Decay} proves optimal decay estimates on the norms $||f(t,\cdot)||_s = ||\ |\xi|^s \hat{f}(t,\xi)||_{L_\xi^1}$, matching the estimates for the linear case \eqref{e:linear}.  Under the weaker assumption that $||f'_0||_{L^\infty}<1$, \cite{Globalold} also showed that a maximum principle holds for the slope and global Lipschitz weak solutions exist.  The authors of \cite{Monotonic} were also able show a maximum principle for the slope and the existence of global weak solutions as well, but under the assumption that the initial data $f_0$ was monotonic rather than slope less than 1.  In our earlier work \cite{MuskatMe}, we connected these two results by proving that there exist global classical solutions whenever the product of the maximal and minimal slope is less than 1, which amounts to an angular condition on the unit normal of the graph.  Using a reformulation of \eqref{e:fequation} and a number of Besov space estimates, \cite{H3/2} develops a $\dot{H}^{3/2}$ critical theory for the Muskat problem under a bounded slope assumption.  The authors of \cite{ConstantinMain} made great progress towards proving global regularity by improving the existing continuation criteria from the $C^{2,\delta}$ established in \cite{Localwellpose} to $C^1$.  That is, they proved that if the initial data $f_0\in H^k(\R)$, then the solution $f$ will exist and remain in $H^k$ so long as the slope $f'(t,\cdot)$ remains bounded and uniformly continuous.  

In 3D, there are a handful of results.  In \cite{3DSlope} the authors show that global classical solutions exist when the initial data $f_0$ satisfies $|| \ |\xi| \hat{f}||_{L^1_{\xi}(\R^2)}$ is less than some constant $k_0 \approx 1/5$.  Notably, \cite{ViscosityJump} was able to replicate this result even when the viscosity of the two fluids are distinct, which is the first such result in the viscosity jump case.  The authors of \cite{3DSlope} also prove a maximum principle for the slope and global weak solutions  whenever the initial slope $||\nabla_x f_0 ||_{L^\infty(\R^2)}$ is suitably bounded.  In the paper they state the theorem for $||\nabla_x f_0||_{L^\infty}< 1/3$, but a careful reading of their proof of Theorem 4.1 shows that this holds in fact for the improved bound $||\nabla_x f_0||_{L^\infty} < 5^{-1/2}$ .  Under this assumption of slope less than $5^{-1/2}$, we proved the Muskat problem was globally wellposed, with global classical solutions obeying the comparison principle in \cite{3DMuskatMe}.  Recently, \cite{3DH2} was able to prove global well-posedness in the critical space $\dot{H}^2\cap \dot{W}^{1,\infty}$ whenever $||f_0||_{\dot{H}^2}\leq \mathcal{F}(||\nabla_x f_0||_{L^\infty})$, thus allowing for arbitrarily large slopes so long as the $H^2$ norm is sufficiently small.

\section{Main Results}

\begin{theorem}\label{t:title}
Fix some Lipschitz constant $2\leq L<\infty$.  Then there is a time $T^*=T^*(L)$  such that the following holds: if $f:[0,T^*]\times \R^{2}\to [0,1]$ is a classical, periodic solution to the stable Muskat problem \eqref{e:fequation} with $||\nabla_x f||_{L^\infty}\leq L$, then $f$ extends to a unique classical solution $f:[0,\infty)\times \R^2\to \R$ to \eqref{e:fequation}.  The time $T^*$ grows at most polynomially in $L$, with 
\begin{equation}
T^*\leq C L^5 \ln(L).
\end{equation}

\end{theorem}

Our strategy for proving Theorem \ref{t:title} is to show that any classical solution flattens over time, and then apply the globally well-posedness result of \cite{3DMuskatMe}.  As the solution $f$ is bounded, on large scales it already appears very flat.  The Lipschitz bound guarantees that \eqref{e:fequation} is a uniformly elliptic equation, and we use this ellipticity to show that flatness propagates from large scales to progressively smaller scales.

For our proof of flattening, we take inspiration from Kiselev's proof of eventual H\"older regularization for supercritical Burger's equation \cite{KiselevSurvey}.   In order to measure this improvement in flatness, we define a modulus $\omega: [0,\infty)^2\to [0,\infty)$ by 
\begin{equation}\label{e:omegadefn}
\omega(t,r) = j(t) + \left\{\begin{array}{ll} \nu(L)(r - \frac{r^{3/2}}{2^{3/2}}), & 0\leq r\leq 2, \\ \frac{\nu(L)}{2}r, & 2\leq r \leq \frac{2}{\nu(L)}, \\ 1, & r\geq \frac{2}{\nu(L)}, \end{array}\right. 
\end{equation}
where the constant $\nu(L)$ is given by 
\begin{equation}
\nu(L) = \frac{L}{3(L^2+1)^{3/2}},
\end{equation}
and the function $j: [0,\infty)\to [0,1]$ is a decreasing function with $j(0)=1$ and $j(T^*) = 0$.  The function $j(t)$ essentially measures how far in the $L^\infty$ norm we allow our function $f(t,\cdot)$ to be from a $\nu(L)$-Lipschitz function.  
See Lemma \ref{l:jdefn} for the precise definition of $j(t)$ and the time $T^*$.  

\begin{theorem}\label{t:main}
Let $f:[0,T^*]\times \R^2\to [0,1]$ be a classical, periodic solution to the Muskat equation \eqref{e:fequation} with $||\nabla_x f||_{L^\infty}\leq L$.  Then 
\begin{equation}
f(t,x)-f(t,y)\leq \omega(t,|x-y|), \qquad \forall t\in [0,T^*], \ x,y\in \R^2.  
\end{equation}
In particular, $f(T^*,\cdot)$ is a $\nu(L)$-Lipschitz function.  
\end{theorem}

As $\nu(L) < 5^{-1/2}$, Theorem \ref{t:main} combined with the global well-posedness result from \cite{3DMuskatMe} thus proves Theorem \ref{t:title}.

\begin{remark}
We note that while for technical reasons we require $f$ to be periodic in Theorem \ref{t:title} and \ref{t:main}, the time $T^*$ is independent of the exact period, and the theorem should be true for general bounded solutions.  Periodicity is only used in the proof of Lemma \ref{l:breakthrough}, and any other assumption that allows us to guarantee a first crossing point will also suffice.  
\end{remark}

\section{Necessary Lemmas}

\begin{lemma}\label{l:breakthrough}(Breakthrough Argument)
Let $f:[0,T^*]\times \R^{2}\to [0,1]$ be a classical, periodic solution of the Muskat equation \eqref{e:fequation}.  Then either 
\begin{equation} \label{e:controloverf}
f(t,x)-f(t,y) \leq \omega(t,|x-y|), \qquad \forall t\in [0,T^*], x,y\in \R^{2},
\end{equation}
or there exists a first crossing point.  I.e., there is some time $t_0\in (0,T^*)$ and points $x_0\not = y_0\in \R^{2}$ such that 
\begin{equation}\label{e:crossingpointassumptions}
\left\{\begin{array}{l}\delta_h f(t_0,x)\leq \omega(t_0, |h|), \qquad x,h\in \R^{2}, \\ f(t_0,x_0)-f(t_0,y_0) = \omega(t_0, |x_0-y_0|), \\ \partial_t \omega(t_0, |x_0-y_0|)\leq \partial_t f(t_0,x_0) - \partial_t f(t_0,y_0).\end{array}\right.
\end{equation}
\end{lemma}

\begin{proof}
Suppose that \eqref{e:controloverf} does not hold, and define the time $t_0$ by 
\begin{equation}
\sup\{t\in [0,T^*]: f(s,x)-f(s,y)< \omega(s,|x-y|), \quad \forall s\leq t, x,y\in \R^{2}\}. 
\end{equation}
Note that at time $t=0$, we have trivially by our $L^\infty$ bounds that for any $x, y\in \R^{2}$ 
\begin{equation}
f(0,x) - f(0,y) < \omega(0,|x-y|).
\end{equation}
As $f$ is a continuous periodic function, we thus have that there is some $\epsilon>0$ such that 
\begin{equation}
f(0,x) - f(0,y)\leq \epsilon +\omega(0,|x-y|).
\end{equation}
Thus by uniform continuity, the same will hold for sufficiently small times $t$.  Thus $t_0\in (0,T^*]$ is well defined.  And as we are assuming \eqref{e:controloverf} does not hold, we necessarily have that $t_0\in (0,T^*)$.  
Because $f$ is continuous, at time $t_0$ we must have that for all $x,y\in \R^{2}$
\begin{equation}
f(t_0,x)-f(t_0,y)\leq \omega(t_0,|x-y|).
\end{equation}
Following the same logic that we used at time $t=0$, then by the definition of $t_0$ we must have that there exist two points $x_0,y_0\in \R^{2}$ such that 
\begin{equation}
f(t_0,x_0) - f(t_0,y_0)  = \omega(t_0,|x_0-y_0|).
\end{equation}
As $f$ is a well defined function, clearly $x_0\not = y_0$.  And by the definition of $t_0$ we have that for any $\epsilon>0$, 
\begin{equation}
f(t_0-\epsilon,x_0) - f(t_0-\epsilon,y_0) < \omega(t_0-\epsilon, |x_0-y_0|).
\end{equation}
Thus 
\begin{equation}
\partial_t \omega(t_0,|x_0-y_0|)\leq \partial_t f(t_0,x_0) - \partial_t f(t_0,y_0).
\end{equation}

\end{proof}

\begin{lemma}\label{l:growthbounds}
Let $f:[0,T^*]\times\R^{2}\to \R$ be a classical $L-$Lipschitz solution to the Muskat equation \eqref{e:fequation}.  Suppose that the crossing point assumptions \eqref{e:crossingpointassumptions} hold for the function $\omega$ defined in \eqref{e:omegadefn}.  Then 
\begin{equation}
\left\{\begin{array}{l} \delta_h f(t_0,x_0)\leq \delta_h f(t_0,y_0),   \ \qquad\qquad \forall h\in \R^2,
\\-L|h|\leq \delta_h f(t_0,x_0) \leq \nu(L)|h|,  \qquad \forall h\in \R^{2},
\\ - \nu(L)|h| \leq \delta_h f(t_0,y_0) \leq L|h|,  \qquad \forall h\in \R^{2}, 
\\ \nabla_x f(t_0,x_0) = \nabla_x f(t_0,y_0) = \partial_r \omega(t_0,|x_0-y_0|) \frac{x_0-y_0}{|x_0-y_0|}.  \end{array}\right.  
\end{equation}
\end{lemma}

\begin{proof}
By our $L$-Lipschitz assumption, we already have for free that 
\begin{equation}
|\delta_h f(t_0,x)| \leq L |h|, \qquad \forall x,h\in \R^{2}.
\end{equation}
To see the stronger upper bound for $\delta_h f(t_0,x_0)$, note that by our crossing point assumptions \eqref{e:crossingpointassumptions} that 
\begin{equation}
\begin{split}
\delta_h f(t_0,x_0) &= f(t_0,x_0+h)-f(t_0,x_0) = (f(t_0,x_0+h)-f(t_0,y_0)) - (f(t_0,x_0)-f(t_0,y_0)) 
\\& \leq \omega(t_0,|x_0+h-y_0|) - \omega(t_0,|x_0-y_0|)\leq \nu(L)|h|.
\end{split}
\end{equation}
The lower bound for $\delta_h f(t_0,y_0)$ follows by the same argument.  Similarly, 
\begin{equation}
\delta_h f(t_0,x_0) - \delta_h f(t_0,y_0)  = (f(t_0,x_0+h) - f(t_0, y_0+h)) - \omega(t_0,|x_0-y_0|) \leq 0.
\end{equation}
Finally, as the map $f(t_0,\cdot)$ is touched from above at $x_0$ by the map 
$x\to \omega(t_0, |x-y_0|) +f(t_0,y_0)$ and from below at $y_0$ by that map $y\to f(t_0,x_0)-\omega(t_0, |x_0-y|)$, we get the equality of the gradients above.  
\end{proof}

\begin{lemma}\label{l:monotonicity}
Fix some constant $L\geq 1$, and define $\nu(L)$ by 
\begin{equation}
\nu(L) = \frac{1}{3}\frac{L}{(L^2+1)^{3/2}}.
\end{equation}
Let $\overline{\alpha}, \underline{\alpha}, \beta\in \R$ be such that the following inequalities hold:
\begin{equation}
\left\{\begin{array}{l} \underline{\alpha} \leq \overline{\alpha}, \\ -L\leq \underline{\alpha} \leq \nu(L), \\ -\nu(L)\leq \overline{\alpha}\leq L, \\ |\beta|\leq \nu(L). \end{array}\right. 
\end{equation}
Then 
\begin{equation}\label{e:alphainequality}
\frac{\overline{\alpha} + \beta}{(\overline{\alpha}^2 + 1)^{3/2}} - \frac{\underline{\alpha} + \beta}{(\underline{\alpha}^2 + 1)^{3/2}} \geq \frac{1}{3}\frac{\overline{\alpha}  - \underline{\alpha} }{(L^2+1)^{3/2}}.
\end{equation}
\end{lemma}

\begin{proof}
It is straightforward to check that the difference in \eqref{e:alphainequality} is nonnegative.  The quantity
\begin{equation}
\frac{1}{\overline{\alpha}  - \underline{\alpha}} \left(\frac{\overline{\alpha} + \beta}{(\overline{\alpha}^2 + 1)^{3/2}} - \frac{\underline{\alpha} + \beta}{(\underline{\alpha}^2 + 1)^{3/2}} \right),
\end{equation}
is minimized when $\overline{\alpha} = L$, $\underline{\alpha} = \nu(L)$, and $\beta = \nu(L)$.  Then 
\begin{equation}
\frac{L+\nu(L)}{(L^2+1)^{3/2}} - \frac{2\nu(L)}{(\nu(L)^2+1)^{3/2}} \geq \frac{L}{(L^2+1)^{3/2}} - 2\nu(L) \geq \frac{1}{3} \frac{L}{(L^2+1)^{3/2}} >\frac{1}{3}\frac{L-\nu(L)}{(L^2+1)^{3/2}}.
\end{equation}
\end{proof}

\begin{lemma}\label{l:integralfbound}
Let $f:[0,T^*]\times \R^{2}\to \R$ be an $L$-Lipschitz classical solution to the Muskat equation \eqref{e:fequation} that satisfies the crossing point assumptions \eqref{e:crossingpointassumptions}.  Then 
\begin{equation}
\partial_t f(t_0,x_0) - \partial_t f(t_0,y_0) \leq \frac{1}{3(L^2+1)^{3/2}} \int\limits_{\R^{2}} \frac{\delta_h f(t_0,x_0) - \delta_h f(t_0,y_0)}{|h|^{d}} dh.
\end{equation}
\end{lemma}
\begin{proof}
Simply use the Muskat equation \eqref{e:fequation} and apply Lemma \ref{l:growthbounds} and Lemma \ref{l:monotonicity} with 
\begin{equation}
\underline{\alpha} = \frac{\delta_h f(t_0,x_0)}{|h|}, \quad \overline{\alpha} = \frac{\delta_h f(t_0,y_0)}{|h|}, \quad \beta =  \partial_r \omega(t_0,|x_0-y_0|) \frac{(x_0-y_0)}{|x_0-y_0|}\cdot\frac{h}{|h|}.
\end{equation}

\end{proof}

\begin{lemma}\label{l:ftoomegabound} 
Let $f$ satisfy the crossing point assumptions \eqref{e:crossingpointassumptions}, and $\xi = |x_0-y_0|>0$.  Then suppressing the time variable, we have that
\begin{equation}
\int\limits_{\R^{2}}\frac{\delta_h f(x_0)-\delta_h f(y_0)}{|h|^d}dh \leq \int\limits_{0}^{\xi/2} \frac{\omega(\xi + 2\eta)+\omega(\xi-2\eta) - 2\omega(\xi)}{\eta^2}d\eta + \int\limits_{\xi/2}^\infty \frac{\omega(2\eta+ \xi ) - \omega( 2\eta-\xi) - 2\omega(\xi)}{\eta^2}d\eta .
\end{equation}
\end{lemma}
Lemma \ref{l:ftoomegabound} is due to \cite{KiselevIntRearrange}, and its proof can be found in the appendix of that paper.

\begin{lemma}\label{l:integralomegabound}
Let $\omega$ be as in \eqref{e:omegadefn}.  Then for any $0<\xi < \displaystyle\frac{2}{\nu(L)}$, we have that 
\begin{equation}
\begin{split}
 \int\limits_{0}^{\xi/2} \frac{\omega(\xi + 2\eta)+\omega(\xi-2\eta) - 2\omega(\xi)}{\eta^2}d\eta &+ \int\limits_{\xi/2}^\infty \frac{\omega(2\eta+ \xi ) - \omega( 2\eta-\xi) - 2\omega(\xi)}{\eta^2}d\eta  
 \\&< - \min\left\{\nu(L)\omega(0) + \frac{\nu(L)^2}{2}, \omega(0)^{1/3}\nu(L)^{2/3}\right\}
 \end{split}
\end{equation}
\end{lemma}

\begin{proof}
Suppose that $\xi\geq 1$.  Note that by the concavity of $\omega$ in $r$, both integrals 
\begin{equation}
 \int\limits_{0}^{\xi/2} \frac{\omega(\xi + 2\eta)+\omega(\xi-2\eta) - 2\omega(\xi)}{\eta^2}d\eta, \int\limits_{\xi/2}^\infty \frac{\omega(2\eta+ \xi ) - \omega( 2\eta-\xi) - 2\omega(\xi)}{\eta^2}d\eta  <0.
\end{equation}
Thus it suffices for us to bound the second integral.  By restricting our integral to $\eta\geq \displaystyle\frac{2}{\nu(L)}> \frac{1}{\nu(L)}+\frac{\xi}{2}$, we have that 
\begin{equation}
 \int\limits_{\xi/2}^\infty \frac{\omega(2\eta+ \xi ) - \omega( 2\eta-\xi) - 2\omega(\xi)}{\eta^2}d\eta  <\int\limits_{2/\nu(L)}^\infty \frac{-2\omega(\xi)}{\eta^2}d\eta \leq \omega(\xi)\nu(L) \leq \omega(0)\nu(L)+ \frac{\nu(L)^2}{2}.
\end{equation}

Now suppose that $0<\xi \leq 1$.  Using that 
\begin{equation}
\omega(2\eta+ \xi ) - \omega( 2\eta-\xi) - 2\omega(\xi) \geq -2\omega(0),
\end{equation}
we get immediately that 
\begin{equation}\label{e:reallysmallxi}
 \int\limits_{\xi/2}^\infty \frac{\omega(2\eta+ \xi ) - \omega( 2\eta-\xi) - 2\omega(\xi)}{\eta^2}d\eta \leq \frac{-4\omega(0)}{\xi}.
\end{equation}
As for the other integral, rescaling gives us immediately that 
\begin{equation}\label{e:lesssmallxi}
 \int\limits_{0}^{\xi/2} \frac{\omega(\xi + 2\eta)+\omega(\xi-2\eta) - 2\omega(\xi)}{\eta^2}d\eta  =\frac{-\nu(L)\sqrt{\xi}}{2^{3/2}}\int\limits_0^{1/2} \frac{(1+2\eta')^{3/2} + (1-2\eta')^{3/2}-2}{(\eta')^2} d\eta' < \frac{-3\nu(L)\sqrt{\xi}}{2^{5/2}}.
\end{equation}
If $\omega(0)\geq 8\nu(L)$, then \eqref{e:reallysmallxi} suffices.  Else, applying \eqref{e:lesssmallxi} for $1\geq \xi\geq \displaystyle\frac{\omega(0)^{2/3}}{4\nu(L)^{2/3}}$ and \eqref{e:reallysmallxi} for the rest gives 
\begin{equation}
\int\limits_{0}^{\xi/2} \frac{\omega(\xi + 2\eta)+\omega(\xi-2\eta) - 2\omega(\xi)}{\eta^2}d\eta + \int\limits_{\xi/2}^\infty \frac{\omega(2\eta+ \xi ) - \omega( 2\eta-\xi) - 2\omega(\xi)}{\eta^2}d\eta  < -(\omega(0)\nu(L)^2)^{1/3}
\end{equation}
\end{proof}

\begin{lemma}\label{l:jdefn}
Define $j: [0,\infty)\to [0,1]$ to be the unique, nonnegative solution of the initial value problem
\begin{equation}
j'(t) = \left\{\begin{array}{ll} \frac{-\nu(L)^2}{L}j(t), & \frac{\nu(L)}{2}< j(t)\leq 1 ,  \\ \frac{-\nu(L)^3}{2L}, & \frac{\nu(L)^4}{8} < j(t) \leq \frac{\nu(L)}{2},  
\\ -\frac{\nu(L)^{5/3}}{L}j(t)^{1/3}, &  0\leq j(t)\leq \frac{\nu(L)^4}{8} \end{array}\right.
\end{equation}
with $j(0)=1$.  
Then $j \equiv 0$ after some finite time $T^*$, with 
\begin{equation}
T^* := \sup\{ t| j(t)>0\} = \frac{L\ln(2/\nu(L))}{\nu(L)^2} + \left(\frac{\nu(L)}{2} - \frac{\nu(L)^4}{8}\right)\frac{2L}{\nu(L)^3} + \frac{3}{8}\nu(L)L .
\end{equation}
Furthermore, $j\in C^1([0,T^*))$ and for any time $t\in (0,T^*)$ 
\begin{equation}
j'(t) \geq - \min\left\{\frac{\nu(L)^2}{L}j(t) + \frac{\nu(L)^3}{2L},\frac{\nu(L)^{5/3}}{L} j(t)^{1/3}\right\}.
\end{equation}
\end{lemma}

The lemma easily follows by the definition of the function $j(t)$.  Note that as $\nu(L)\sim L^{-2}$, it follows that for $L\geq 2$ 
\begin{equation}
T^* \lesssim L^5\ln(L).
\end{equation}

\section{Proof of Theorem \ref{t:main}}

Fix some Lipschitz constant $L<\infty$.  Take $T^*$ and $j:[0,\infty)\to [0,1]$ to be as in Lemma \ref{l:jdefn}, and the modulus $\omega:[0,\infty)^2\to [0,\infty)$ be as in \eqref{e:omegadefn}.  Let $f:[0,T^*]\times \R^2\to [0,1]$ be a classical, periodic solution to the Muskat equation \eqref{e:fequation} with $||\nabla_x f||_{L^\infty}\leq L$.

We wish to show that for all $t\in [0,T^*]$ and $x,y\in \R^2$ that 
\begin{equation}\label{e:maingoal}
f(t,x)-f(t,y)\leq \omega(t,|x-y|).
\end{equation}
Suppose that \eqref{e:maingoal} was false.  Then by the break through argument Lemma \ref{l:breakthrough}, we have that there exists a time $t_0\in (0,T^*)$ and points $x_0,y_0\in \R^2$ such that 
\begin{equation}\label{e:crossingpointassumptions2}
\left\{\begin{array}{l}\delta_h f(t_0,x)\leq \omega(t_0, |h|), \qquad x,h\in \R^{2}, \\ f(t_0,x_0)-f(t_0,y_0) = \omega(t_0, |x_0-y_0|), \\ \partial_t \omega(t_0, |x_0-y_0|)\leq \partial_t f(t_0,x_0) - \partial_t f(t_0,y_0).\end{array}\right.
\end{equation}
Note that as $\omega(t_0,r)>1$ for $r\geq \displaystyle\frac{2}{\nu(L)}$, we must necessarily have that $|x_0-y_0|< \displaystyle\frac{2}{\nu(L)}$.  Applying the lemmas \ref{l:integralfbound}, \ref{l:ftoomegabound}, and \ref{l:integralomegabound} along with \eqref{e:crossingpointassumptions2} then gives us that 
\begin{equation}
j'(t_0) = \partial_t \omega(t_0,|x_0-y_0|) < - \min\left\{\frac{\nu(L)^2}{L}j(t_0) + \frac{\nu(L)^3}{2L},\frac{\nu(L)^{5/3}}{L} j(t_0)^{1/3}\right\}.
\end{equation}
But by construction, we have that the function $j$ satisfies 
\begin{equation}
j'(t)\geq - \min\left\{\frac{\nu(L)^2}{L}j(t) + \frac{\nu(L)^3}{2L},\frac{\nu(L)^{5/3}}{L} j(t)^{1/3}\right\},
\end{equation}
a contradiction.  Thus \eqref{e:maingoal} must hold, and Theorem \ref{t:main} is proven.  

\section*{Acknowledgements}  
The author was partially supported by an NSF postdoctoral fellowship, NSF DMS 1902750.

\bibliography{Muskat-References}{}
\bibliographystyle{alpha}

\end{document}